\documentclass{article}
\usepackage{cite}
\usepackage{amsmath,amssymb,amsthm,mathrsfs}
\usepackage{titlesec,hyperref}
\usepackage{color}
\usepackage{fancyhdr}
\usepackage[margin=2.5cm]{geometry}
\usepackage{enumerate}
\usepackage[integrals]{wasysym}
\usepackage{bm}
\usepackage{graphicx}
\usepackage{caption}
\usepackage{float}
\usepackage{subfigure}

\newtheorem{theo}{Theorem}[section]
\newtheorem{lemm}[theo]{Lemma}
\newtheorem{defi}[theo]{Definition}

\newtheorem{rema}[theo]{Remark}
\def\ud{{\rm d}}
\numberwithin{equation}{section}

\allowdisplaybreaks

\begin{document}

\title{The existence and uniqueness of global admissible conservative weak solution for the periodic single-cycle pulse equation}
\author{
    Yingying $\mbox{Guo}^1$ \footnote{email: guoyy35@fosu.edu.cn} \quad and\quad
    Zhaoyang $\mbox{Yin}^{2,3}$ \footnote{email: mcsyzy@mail.sysu.edu.cn}\\
    $^1\mbox{Department}$ of Mathematics, Foshan University,\\
    Foshan, 528000, China\\
    $^2\mbox{Department}$ of Mathematics, Sun Yat-sen University,\\
    Guangzhou, 510275, China\\
    $^3\mbox{Faculty}$ of Information Technology,\\
    Macau University of Science and Technology, Macau, China
}
\date{}
\maketitle
\begin{abstract} 
  This paper is devoted to the study of the existence and uniqueness of global admissible conservative weak solutions for the periodic single-cycle pulse equation. We first transform the equation into an equivalent semilinear system by introducing a new set of variables. Using the standard ordinary differential equation theory, we then obtain the global solution to the semilinear system. Next, returning to the original coordinates, we get the global admissible conservative weak solution for the periodic single-cycle pulse equation. Finally, given an admissible conservative weak solution, we find a equation to single out a unique characteristic curve through each initial point and prove the uniqueness of global admissible conservative weak solution without any additional assumptions.
\end{abstract}

\tableofcontents

\section{Introduction}
In this paper, we consider the initial problem for the following periodic integrable single-cycle pulse equation
\begin{equation}\label{eq1}
\left\{\begin{array}{lll}
u_{xt}=u+\frac{1}{2}u(u^2)_{xx},&\quad t\geq0,\quad x\in\mathbb{R},\\
u(t,x)|_{t=0}=u_0(x),&\quad x\in\mathbb{R},\\
u(t,x+1)=u(t,x),&\quad t\geq0,\quad x\in\mathbb{R}.
\end{array}\right.
\end{equation}
Indeed, Eq. \eqref{eq1} is a generalized short pulse equation \cite{s}
\begin{equation}\label{gsp}
u_{xt}=u+au^2u_{xx}+buu_{x}^2,
\end{equation}
with $a/b=1$.
In \cite{s}, Sakovich stated the generalized short pulse equation \eqref{gsp} is integrable in two cases of its coefficients. The first one is the short pulse (SP) equation \cite{sw}
\begin{equation}\label{sp}
u_{xt}=u+\frac{1}{6}(u^3)_{xx}
\end{equation}
with $a/b=1/2$. The SP equation \eqref{sp} is a model as an alternative equation to the cubic nonlinear Schr\"{o}dinger (NLS) equation to describe the evolution of very short optical pulses in nonlinear media \cite{sw}. It is integrable \cite{fmo,ss1} and has the Lax pair, bi-Hamiltonian structure \cite{b2}. The local well-posedness, global well-posedness and blow-up phenomenon were studied in \cite{lps,ps}. There are many other lectures for SP equation, see \cite{b1,m1,m2,p1,p2,ss2}.

The second one is the above single-cycle pulse equation \eqref{eq1} with $a/b=1$ for the reason the smooth envelope soliton of Eq. \eqref{eq1} is as short as one cycle of its carrier frequency \cite{s}. Thanks to the soliton theory, there are many interesting properties of Eq. \eqref{eq1} to investigate. For instance, the well-posedness, blow-up solutions, periodic solutions, weak solutions, and so on. In effect, Li and Yin \cite{ly2} established the local well-posedness of Eq. \eqref{eq1} in $H^s(\mathbb{S})$ with $s\geq2$ by Kato theory. Moreover, they derived a relationship between the single-cycle pulse equation and the sine-Gordon equation and finally got a global existence result by useful conservative quantities. However, they didn't give any results about the weak solutions for Eq. \eqref{eq1}.

Recently, Hone, Novikov and Wang \cite{hnw} displayed a more general equation than Eq. \eqref{gsp} by classifying the nonlinear partial differential equations of second order
\begin{align}\label{mgsp}
u_{xt}=u+c_0 u^2 + c_1 uu_x + c_2 uu_{xx} + c_3 u^2_x + d_0 u^3+ d_1 u^2 u_x + d_2 u^2 u_{xx} + d_3 uu^2_x.
\end{align}
In fact, Eq. \eqref{mgsp} not only includes Eq. \eqref{eq1} and Eq. \eqref{sp}, but also contains another vital equation that we called the dispersive Hunter-Saxton equation
\begin{align}\label{dhs}
u_{xt}=u+2uu_{xx}+u_x^2
\end{align}
for the reason Eq. \eqref{dhs} has one more dispersive term $u$ than the following Hunter-Saxton (HS) equation \cite{hs}
\begin{align}\label{hs}
(u_t+uu_x)_x=\frac12u_x^2.
\end{align}
The local well-posedness and travelling wave solutions of the periodic dispersive Hunter-Saxton equation \eqref{dhs} were studied in \cite{ly1}. In addition, the HS equation \eqref{hs} as an asymptotic model of liquid crystals is local well-posed and has global strong solutions, global weak solutions and global  dissipative solutions, see \cite{bc1,bss,lz,y}.

The HS equation \eqref{hs} can as well arise in a different physical context as the high-frequency limit \cite{dp,hz} of the following classical Camassa-Holm (CH) equation
\begin{equation}\label{chh}
u_t-u_{xxt}+3uu_{x}=2u_{x}u_{xx}+uu_{xxx}
\end{equation}
which is completely integrable \cite{c2,cgi,cmc} and has a bi-Hamiltonian structure \cite{ff}. The solitary waves, peak solitons, local well-posedeness and ill-posedeness, global strong solutions and blow-up strong solutions of the CH equation were discussed in \cite{c1,c3,c4,c5,ce1,ce2,ce3,ce4,ce5,cs,d,glmy,l,t}. In addition, the CH equation has global conservative weak solutions and dissipative weak solutions \cite{chk,xz1,xz2}. Moreover, the existence and uniqueness of global conservative solutions on the line were studied in \cite{bc2,bc3,bcz,cmo}. And the existence and uniqueness of global conservative solutions on the circle were investigated in \cite{hr,ghr}. However, their unique results depend on a Lipschitz metric $d_{\mathcal{D}}$ from $\mathcal{H}$ to $\mathcal{D}$, where $\mathcal{H}$ satisfies $\int_{0}^{1}y(\xi){\ud}\xi =0$ and $y_{\xi}+H_{\xi}=1+\|H_{\xi}\|_{L^1}$, which means some additional assumptions are needed for the uniqueness.

As far as we know, the existence and uniqueness of global conservative weak solutions for \eqref{eq1} has not been investigated yet. Therefore, in this paper, we aim to study the existence and uniqueness of global admissible conservative weak solutions to \eqref{eq1} by following the idea of Bressan and Constantin \cite{bc2,bcz}. Unlike \cite{ghr}, our proof does not need any additional assumptions.

The rest of our paper is as follows. In the second section, we give some basic equations about the periodic single-pulse equation and present our main results. In the third section, we deduce an equivalent semilinear system by introducing a new set of variables and then establish the global solutions to the semilinear system. In the fourth section, returning to the original variables, we obtain the global admissible conservative weak solution of the original equation. In the last section, by establishing an ordinary differential system, we prove that the global admissible conservative weak solution for \eqref{eq1} is unique.


\section{The basic equations and main results}
In this subsection, we give the basic equations and our main results.
Before that, we first introduce some definitions.
\begin{defi}\cite{st}
	Let $\mathrm{T}_n$ denote a circle of unit length. We say that $\mathcal{D}(\mathrm{T}_n)$ is the collection of all complex-valued infinitely differentiable functions on $\mathrm{T}_n$ if the locally convex topology in it is generated by the semi-norms
	\begin{align*}
		\|v\|_{\beta}=\sup_{x\in \mathrm{T}_n}\left|D^{\beta}v(x)\right|,
	\end{align*}
	where $$\mathrm{T}_n=\left\{~x~|~x\in\mathbb{R}_n,~x=(x_1,\cdots,x_n),~|x_i|\leq\pi,~i=1,\cdots,n\right\}$$ and $\beta=(\beta_1,\cdots,\beta_n)$ is an arbitrary multi-index with non-negative components.
\end{defi}
\begin{rema}\cite{e,st}
	Any function $v(x)\in\mathcal{D}(\mathrm{T}_n)$ can be represented as
	\begin{equation*}
		v(x)=\sum_{k\in \mathbb{Z}_n}a_ke^{ikx}\ \ \left(~convergence\ in\ \mathcal{D}(\mathrm{T}_n)~\right),
	\end{equation*}
	where $\{a_k\}_{k\in \mathbb{Z}_n}$ is a sequence of complex numbers such that
	\begin{equation}\label{per}
		|a_k|\leqq c_m(1+|k|)^{-m},\ k\in \mathbb{Z}_n,
	\end{equation}
	for all $m\in0,1,2,\cdots$. Here $c_m$ is an appropriate positive constant. It holds $a_k=\hat{v}(k),k\in \mathbb{Z}_n$. Conversely, if $\{a_k\}_{k\in \mathbb{Z}_n}$ satisfies \eqref{per} then $\sum\limits_{k\in \mathbb{Z}_n}a_ke^{ikx}$ convergence in $\mathcal{D}(\mathrm{T}_n)$. If $v(x)$ is its limit function then we have $\hat{v}(k)=a_k,k\in\mathbb{Z}_n$.
\end{rema}

For the sake of simplicity, we hereafter assume that the period $\mathrm{T}_1$ is the unit period $\mathbb S$ in one dimension.\\
\textbf{Notation.}
\begin{align*}
	\overline{w}=\int_{\mathbb{S}}w(x)dx~~~~~~~~\ &:\ the\ mean\ value\ of\ the\ real\ function\ w(x)\ over\ \mathbb{S}. ~~~~~~~~~~~~~~~~~~~~~~~~~~~~\nonumber\\
	\mathbb Pw(x)=w(x)-\overline{w}~~~~~~~~~~\ &:\ the\ orthogonal\ projection\ onto\ mean\ zero\ functions.~~~~~~~~~~~~~~~~~~~~~~~~~~~~\nonumber\\
	\partial_x^{-1}w(x)=\int_{0}^{x}{\mathbb P}(w)(y)dy\ &:\  the\ inverse\ of\ the\ differential\ operator.~~~~~~~~~~~~~~~~~~~~~~~~~~~~~~~~~~~~~
\end{align*}

By applying the orthogonal projection to \eqref{eq1}, we obtain
\begin{equation}\label{eq2}
\left\{\begin{array}{lll}
(u_t-u^2u_x)_x=u-uu_x^2-\overline{u-uu_x^2},&\quad t\geq 0,~~x\in \mathbb{R},\\
u(t,x)|_{t=0}=u_0(x),&\quad x\in \mathbb{R},\\
u(t,x+1)=u(t,x),&\quad t\geq0,~~x\in \mathbb{R}.
\end{array}\right.
\end{equation}
Integrating both sides of the first equation in \eqref{eq2} with respect to $x$, and choosing a specific boundary term, we have
\begin{equation}\label{vis}
\left\{\begin{array}{lll}
u_t-u^2\partial_xu=\partial_x^{-1}(u-uu_x^2)-f(t),&\quad t\geq 0,~~x\in \mathbb{R},\\
u(t,x)|_{t=0}=u_0(x),&\quad x\in \mathbb{R},\\
u(t,x+1)=u(t,x),&\quad t\geq0,~~x\in \mathbb{R}.
\end{array}\right.
\end{equation}
Here the boundary term
\begin{equation}\label{h}
f(t)=\frac{1}{1-h}\int_{\mathbb{S}}(1-u_x^2)\partial_x^{-1}(u-uu_x^2)(y){\ud}y
\end{equation}
is carefully selected such that the quantity $\int_{\mathbb{S}}u-uu_x^2{\ud}x $ is conserved when $h=\int_{\mathbb{S}}u_{0x}^2 {\ud}x\neq1$.\\
For smooth solutions, multiplying the first equation of \eqref{eq2} by $2u_x$, we get
\begin{equation}\label{cons}
(u_x^2)_t+\big(-u^2u_x^2\big)_x=\big(u^2\big)_x-2u_x\overline{u-uu_x^2}.
\end{equation}
Combing \eqref{vis} and \eqref{cons}, we deduce that the following quantities
\begin{align}
&E(t):=\int_{\mathbb{S}}u^2_x(t,x){\ud}x,\label{cons1}\\
&F(t):=\int_{\mathbb{S}}(u-uu_x^2)(t,x){\ud}x\label{cons2}
\end{align}
are constants in time.

\begin{defi}\label{defi}
Let $u_0\in H^1(\mathbb{S})$. We say $u(t,x)\in L^\infty\big(\mathbb{R}^+;H^{1}(\mathbb{S})\big)$ is a global conservative weak solution to the Cauchy problem \eqref{vis}, if $u(t,x)$ satisfies
\begin{align}\label{defi1}
\int_{\mathbb{R}^+}\int_{\mathbb S}\big(u\psi_{tx}+u^2u_x\psi_x\big)(t,x){\ud}x{\ud}t=\int_{\mathbb{R}^+}\int_{\mathbb S}\Big[\big(u-uu^2_{x}-\overline{u-uu_x^2}\big)\psi\Big](t,x){\ud}x{\ud}t+\int_{\mathbb{S}}u_{0x}(x)\psi(0,x){\ud}x
\end{align}
and
\begin{align}\label{globalweak1}
\int_{\mathbb{R}^+}\int_{\mathbb S}\big(u\psi_{t}-\frac{1}{3}u^3\psi_x\big)(t,x){\ud}x{\ud}t=-\int_{\mathbb{R}^+}\int_{\mathbb S} \Big[\big(\partial_x^{-1}(u-uu_x^2)-f(t)\big)\psi\Big](t,x){\ud}x{\ud}t-\int_{\mathbb{S}}u_0(x)\psi(0,x){\ud}x
\end{align}
for all $\psi\in C_{0}^\infty(\mathbb{R}^+;\mathcal{D}(\mathbb S))$. Moreover, the quantities $\int_{\mathbb S}u_x^2(t,x){\ud}x,\ \int_{\mathbb S}\big(u-uu_x^2)(t,x){\ud}x$ are conserved in time.
\end{defi}

\begin{defi}\label{defi3}
Let $u_0\in H^1(\mathbb{S})$. We say $u(t,x)\in L^\infty\big(\mathbb{R}^+;H^{1}(\mathbb{S})\big)$ is a global admissible conservative weak solution to the Cauchy problem \eqref{vis}, if $u(t,x)$ satisfies the following properties.
\begin{itemize}
	\item [1.] The function $u$ provides a solution to Cauchy problem (\ref{vis}) in the sense of Definition \ref{defi}.
	\item [2.] For all $\psi\in C_{0}^\infty(\mathbb{R}^+;\mathcal{D}(\mathbb S))$, $u(t,x)$ satisfies
	\begin{align}\label{globalweak2}
	\int_{\mathbb{R}^+}\int_{\mathbb S}\big(u^2_x\psi_{t}-u^2u^2_x\psi_{x}\big){\ud}z{\ud}t=-\int_{\mathbb{R}^+}\int_{\mathbb S}\big((u^2)_x-2u_x\overline{u-uu_x^2}\big)\psi{\ud}z{\ud}t-\int_{\mathbb{S}}u^2_{0x}(z)\psi(0,z){\ud}z.
	\end{align}
\end{itemize}


\end{defi}

The main theorem of this paper can be stated as follows.
\begin{theo}\label{th}
	Let $u_0(x)\in H^{1}(\mathbb{S})$. Suppose that $u_0(x)$ satisfies
	\begin{align}\label{cond1}
		\int_{\mathbb{S}}u_0-u_0u_{0x}^2{\ud}x=0.
	\end{align}
	Then the problem \eqref{eq1} has a unique global admissible conservative weak solution in the sense of Definition \ref{defi3}.
\end{theo}

\section{An equivalent semilinear system and global solution of the system}
\subsection{An equivalent semilinear system}

Define the spaces $I,\ X$ as
\begin{align*}
&I=\big\{g\in H_{loc}^{1}(\mathbb{R})~~|~~g(\eta+1)=g(\eta)+1,~~\text{for all}~~\eta\in\mathbb{R}\big\},\\
&X=I\times H^1(\mathbb{S})\times L^{\infty}(\mathbb{S})\times L^{\infty}(\mathbb{S})\times L^{\infty}(\mathbb{S})
\end{align*}
with the norms $\|g\|_{I}=\|g\|_{H^1}$ and $\|(y,U,V,W,Q)\|_{X}=\|y\|_{H^1}\times\|U\|_{H^1}\times\|V\|_{L^\infty}\times\|W\|_{L^\infty}\times\|Q\|_{L^\infty}$.

Assume that $u$ is smooth and periodic. Let $y:\mathbb{R}\rightarrow I, t\mapsto y(t,\cdot)$ be the characteristic as the solutions of
\begin{equation}\label{char}
y_t(t,\xi)=-u^2\big(t,y(t,\xi)\big),\ y(0,\xi)=y_0(\xi),
\end{equation}
where $y_0(\xi)$ will be given below.\\
Introduce some new variables
\begin{align*}
&U(t,\xi)=u\big(t,y(t,\xi)\big),\quad V(t,\xi)=\frac{1}{1+u_x^2\circ y},\\
&W(t,\xi)=\frac{u_x\circ y}{1+u_x^2\circ y},\quad Q(t,\xi)=(1+u_x^2\circ y)y_{\xi}.
\end{align*}
Owing to the characteristic \eqref{char} and the first equation of \eqref{vis}, we can deduce
\begin{align*}
U_t(t,\xi)=K(\xi)-P(1)y(\xi)-\frac{1}{1-h}\int_{\mathbb{S}}\big(2QV-Q\big)(\eta)\left(K(\eta)-P(1)y(\eta)\right){\ud}\eta,
\end{align*}
where $K(\xi)=\int_{y^{-1}(t,0)}^\xi\left(2UQV-UQ\right)(\eta){\ud}\eta$ and $P(1)=\int_{\mathbb{S}}(2UQV-UQ)(\zeta){\ud}\zeta$ is a function of time $t$.

Similarly, using \eqref{vis},\ \eqref{cons} and \eqref{char}, we obtain
\begin{align*}
&V_t(t,\xi)=-2UW+2WVP(1),\\
&W_t(t,\xi)=2UV-U-2V^2P(1)+VP(1),\\
&Q_t(t,\xi)=-2WQP(1).
\end{align*}
In a word, we formally derive a equivalent semilinear system to \eqref{vis}
\begin{equation}\label{semi}
\left\{\begin{array}{lll}
y_t(t,\xi)=-U^2,\\
U_t(t,\xi)=K(\xi)-P(1)y(\xi)-\frac{1}{1-h}\int_{\mathbb{S}}\big(2QV-Q\big)(\eta)\big(K(\eta)-P(1)y(\eta)\big){\ud}\eta,\\
V_t(t,\xi)=-2UW+2WVP(1),\\
W_t(t,\xi)=2UV-U-2V^2P(1)+VP(1),\\
Q_t(t,\xi)=-2WQP(1).
\end{array}\right.
\end{equation}
It follows that
\begin{equation}\label{semi1}
\left\{\begin{array}{lll}
y_{\xi t}(t,\xi)=-2UU_{\xi},\\
U_{\xi t}(t,\xi)=UQ(2V-1)-P(1)y_{\xi},\\
V_t(t,\xi)=-2UW+2WVP(1),\\
W_t(t,\xi)=2UV-U-2V^2P(1)+VP(1),\\
Q_t(t,\xi)=-2WQP(1).
\end{array}\right.
\end{equation}

\subsection{Global solution of the semilinear system}

In this subsection, we turn our attention to finding a global solution of system \eqref{semi}. Since the spaces $I,\ X$ are not Banach spaces, in order to get the global solution, we need to construct a suitable Banach space. Set $\gamma=y-Id$ where $Id$ denotes the identity. Then, the map $y\mapsto\gamma$ is obviously a bijection between $I$ and $H^{1}(\mathbb{S})$. Therefore, there is a bijection $(y,U,V,W,Q)\mapsto(\gamma,U,V,W,Q)$ between $X$ and $Y:=H^1(\mathbb{S})\times H^1(\mathbb{S})\times L^{\infty}(\mathbb{S})\times L^{\infty}(\mathbb{S})\times L^{\infty}(\mathbb{S})$ which is a Banach space.
Hence, system \eqref{semi} becomes the following equivalent form
\begin{equation}\label{semili}
\left\{\begin{array}{lll}
\gamma_t(t,\xi)=-U^2,\\
U_t(t,\xi)=K(\xi)-P(1)\big(\gamma+Id\big)(\xi)-\frac{1}{1-h}\int_{\mathbb{S}}\big(2QV-Q\big)(\eta)\Big(K(\eta)-P(1)\big(\gamma+Id\big)(\eta)\Big){\ud}\eta,\\
V_t(t,\xi)=-2UW+2WVP(1),\\
W_t(t,\xi)=2UV-U-2V^2P(1)+VP(1),\\
Q_t(t,\xi)=-2WQP(1).
\end{array}\right.
\end{equation}
We will prove that system \eqref{semili} as the ordinary differential equations is local well-posed in the Banach space $Y$. Before that, we choose a appropriate initial data  $(y_0,U_0,V_0,W_0,Q_0)$ as
\begin{equation}\label{data}
\left\{\begin{array}{lll}
y_0(\xi)+\int_0^{y_0(\xi)}u_{0x}^2{\ud}x=(1+h)\xi,\\
U_0(\xi)=u_0\circ y_0(\xi),\\
V_0(\xi)=\frac{1}{1+u_{0x}^2\circ y_0(\xi)},\\
W_0(\xi)=\frac{u_{0x}\circ y_0(\xi)}{1+u_{0x}^2\circ y_0(\xi)},\\
Q_0(\xi)=(1+u_{0x}^2\circ y_0)y_{0\xi}(\xi)=1+h,
\end{array}\right.
\end{equation}
where $h=\int_{\mathbb{S}}u^2_{0x}(x){\ud}x$.

Note that $u_0\in H^1(\mathbb{S})$. Then the function $y_0(\xi)$ is well-defined as the map $y_0\mapsto y_0(\xi)+\int_0^{y_0(\xi)}u_{0x}^2{\ud}x$ is continuous, strictly increasing. It is straightforward to verify that $y_0(0)=0, y_{0\xi}>0$ and $y_0(\xi)\in I$, which follows that $U_0\in W^{1,\infty}(\mathbb{S}),\ V_0\in L^\infty(\mathbb{S}),\ W_0\in L^\infty(\mathbb{S}),\ Q_0\in {L^\infty(\mathbb{S})}$. Therefore, $$\big(\gamma_0,U_0,V_0,W_0,Q_0\big)\in[W^{1,\infty}(\mathbb{S})]^2\times[L^\infty(\mathbb{S})]^3\subset Y.$$

The following lemma gives the local existence of solution to system \eqref{semili} with initial data \eqref{data}.
\begin{lemm}\label{local}
Let $u_0\in H^{1}({\mathbb{S}})$. Then there exists a time $T>0$ such that problem \eqref{semili}-\eqref{data} has a solution $\big(\gamma(t),U(t),V(t),W(t),Q(t)\big)$ in $L^\infty([0,T];Y)$.
\end{lemm}
\begin{proof}
Note that the initial data $\big(\gamma_0,U_0,V_0,W_0,Q_0\big)\in[W^{1,\infty}(\mathbb{S})]^2\times[L^\infty(\mathbb{S})]^3\subset Y$. In order to prove the local existence, we only need to demonstrate that the right side of \eqref{semili} is Lipschitz continuous on every bounded set $B_M\subset Y$ with
$$B_M=\Big\{~\big(\gamma(t),U(t),V(t),W(t),Q(t)\big)\in Y~|~~~\|\big(\gamma(t),U(t),V(t),W(t),Q(t)\big)\|_Y\leq M\Big\}.$$
Here we just verify the second equation of the right side of \eqref{semili} is Lipschitz continuous, since the others are similar and more easier. Recall that $\xi\in\mathbb{S}$, it is obvious that the map $UQV+\frac{1}{2}Q(1-V)$ is Lipschitz continuous from $B_M$ to $L^2(\mathbb{S})$. Then $K(\xi)$ is Lipschitz continuous from $B_M$ to $H^{1}(\mathbb{S})$. In a similar way, $P(1)\big(\gamma+Id\big)$ and $\int_{\mathbb{S}}[K(\eta)-P(1)(\gamma+Id)(\eta)](QV)(\eta){\ud}\eta$ are Lipschitz continuous from $B_M$ to $H^{1}(\mathbb{S})$. Therefore, the Lipschitz continuity is true. Thus, by virtue of the standard ordinary differential equations theory in Banach spaces, there exists a solution $\big(\gamma(t),U(t),V(t),W(t),Q(t)\big)$ to Cauchy problem \eqref{semili}-\eqref{data} on small time interval $[0,T]$ with $T>0$.
\end{proof}
\begin{rema}\label{locall}
Recall that the map $y\mapsto\gamma$ is a bijection between $I$ and $H^{1}(\mathbb{S})$. Then we know that system \eqref{semi} as well has a local solution $\big(y(t),U(t),V(t),W(t),Q(t)\big)$ in $L^\infty([0,T];X)$ from the above lemma.
\end{rema}
\begin{theo}\label{locall1}
Following \cite{hr} and the above lemma, we see that the local solution $\big(\gamma(t),U(t),V(t),W(t),Q(t)\big)$ also belongs to $L^\infty([0,T];[W^{1,\infty}(\mathbb{S})]^2\times[L^\infty(\mathbb{S})]^3)$ with the initial data $\big(\gamma_0,U_0,V_0,W_0,Q_0\big)$, so that
$$\big(\gamma_{\xi}(t),U_{\xi}(t),V(t),W(t),Q(t)\big)\in[L^\infty(\mathbb{S})]^5$$ is a solution to system \eqref{semi1}.
Moreover, we can assert that $y_\xi \geq 0$ and $\rm{meas}(\mathcal{A})=0,\ \rm{meas}(\mathcal{N})=0$ where
\begin{align*}
\mathcal{A}=\Big\{(t,\xi)\in[0,T]\times\mathbb{R}~\big|~y_{\xi}(t,\xi)=0\Big\},\\
\mathcal{N}=\Big\{t\in[0,T]~\big|~y_{\xi}(t,\xi)=0,\ \xi\in\mathbb{R}\Big\}.
\end{align*}
\end{theo}

Next, we will extend the local solution to the global solution.
\begin{theo}\label{global1}
Let $u_0\in H^{1}({\mathbb{S}})$. Then the local solution $\big(\gamma(t),U(t),V(t),W(t),Q(t)\big)$ to problem \eqref{semili}-\eqref{data} is global. Moreover, $$\big(\gamma(t),U(t),V(t),W(t),Q(t)\big)\in[W^{1,\infty}(\mathbb{S})]^2\times[L^\infty(\mathbb{S})]^3$$
for all time $t\geq0$.
\end{theo}
\begin{proof}
To obtain the global existence, without loss of generality, it suffices to demonstrate the solution $\big(\gamma(t),U(t),V(t),W(t),Q(t)\big)$ is uniformly bounded on any bounded time interval $[0,T]$ with $T>0$.

We first claim that
\begin{align}
&W^2+V^2=V,\quad \text{for a.e. $\xi$,}\quad\label{bound}\\
&y_{\xi}=VQ,\quad U_{\xi}=WQ,\quad \text{for a.e. $\xi$.}\label{bound1}
\end{align}
Taking advantage of \eqref{semi} and \eqref{semi1}, we see
\begin{align}
&\big(W^2+V^2\big)_t=V_t,\label{bou3}\\
&\big(U_{\xi}-WQ\big)_t=P(1)\big(VQ-y_{\xi}\big),\label{bou4}\\
&\big(VQ-y_{\xi}\big)_t=WQ-U_{\xi}.\label{bou5}
\end{align}
Observe that at initial time $t=0$, $W^2_{0\xi}+V^2_{0\xi}=V_{0\xi}$, $U_{0\xi}=W_0Q_0$ and $y_{0\xi}=V_0Q_0$, then with \eqref{bou3}, we can prove \eqref{bound}. If $P(1)=0$, then \eqref{bound1} is obviously true. If not, differentiating \eqref{bou4} with respect to $t$ and then using the elliptic equations theory, we can prove \eqref{bound1}.

Next, we prove the conservative laws. Set $\tilde{E}(t):=\int_{\mathbb{S}}(Q-QV)(t,\xi){\ud}\xi,\ \tilde{F}(t):=\int_{\mathbb{S}}(2UQV-UQ)(t,\xi){\ud}\xi$. It's easy to deduce that $\frac{\ud}{{\ud}t}\tilde{E}=\frac{\ud}{{\ud}t}\tilde{F}=0$, that means $\tilde{E}(t)$ and $\tilde{F}(t)$ remain constants in time. So we have
\begin{align}\label{F}
\tilde{E}(t)&=\tilde{E}(0):=\tilde{E}_0=\int_{\mathbb{S}}(Q_0-Q_0V_0)(\xi){\ud}\xi=\int_{\mathbb{S}}(u^2_{0x}\circ y_0) y_{0\xi}{\ud}\xi=\int_{\mathbb{S}}u^2_{0x}{\ud}x,\\
P(1)&=\tilde{F}(t)=\tilde{F}(0):=\tilde{F}_0=\int_{\mathbb{S}}\big(1-u^2_{0x}\circ y_0\big)(u_0\circ y_0)y_{0\xi}{\ud}\xi=\int_{\mathbb{S}}u_0-u_0u^2_{0x}{\ud}x.
\end{align}

We finally want to prove that the local solution is uniformly bounded on any bounded time interval.
\eqref{bound} implies that
\begin{align}\label{vw}
0\leq V\leq1\quad \text{and}\quad \big|W\big|\leq\frac 1 2,
\end{align}
whence $V(t,\xi),\ W(t,\xi)$ is uniformly bounded in $L^\infty([0,T]\times{\mathbb{S}})$.

Notice that $0<1+h=Q_0(\xi)\in L^\infty({\mathbb{S}})$. By solving the fifth equation of \eqref{semili}, we can find
\begin{align}\label{QQ}
0<Q(t,\xi)=Q_0(\xi)e^{\int_0^t\big(-2WP(1)\big){\ud}{\tau}}\leq(1+h)e^{\tilde{F}_0T}.
\end{align}
Hence, $Q(t,\xi)\in L^\infty([0,T]\times{\mathbb{S}})$.

In addition, \eqref{bound1}, \eqref{vw} and \eqref{QQ} infer that
\begin{align}\label{UU}
&\|U_{\xi}\|_{L^\infty}\leq\frac 1 2(1+h)e^{\tilde{F}_0T},\notag\\
&\|\gamma_{\xi}\|_{L^\infty}\leq\|y_{\xi}\|_{L^\infty}+1\leq(1+h)e^{\tilde{F}_0T}+1.
\end{align}
Moreover, using the conserved quantities $\tilde{E}(t),\ \tilde{G}(t)$ and $\tilde{F}(t)$, we can get the uniform boundedness of $u$. In fact, for any fixed $\xi,\eta\in{\mathbb{S}}$,
\begin{align*}
\int_{\mathbb{S}}\big(U(\xi)-U(\eta)\big)\big(2QV-Q\big)(\eta){\ud}\eta=&U(\xi)\int_{\mathbb{S}}\big(2QV-Q\big)(\eta){\ud}\eta-\int_{\mathbb{S}}(2UQV-UQ)(\eta){\ud}\eta\\
=&U(\xi)(1-h)-\tilde{F}_0.
\end{align*}
On the other hand, we discover
\begin{align*}
\Big|\int_{\mathbb{S}}\big(U(\xi)-U(\eta)\big)\big(2QV-Q\big)(\eta){\ud}\eta\Big|&\leq \Big|\int_{\mathbb{S}}\int_{\eta}^{\xi}U_{\xi}{\ud}\zeta\big(2QV-Q\big){\ud}\eta\Big|\\
&\leq\|U_{\xi}\|_{L^\infty}\big(\|2QV\|_{L^\infty}+\|Q\|_{L^\infty}\big)\\
&\leq\frac 3 2(1+h)^2e^{\tilde{F}_0T}\triangleq B.
\end{align*}
Thus, we conclude that
\begin{align}
|U(\xi)|\leq\frac{B+|\tilde{F}_0|}{|1-h|},~~~~~~~a.e.~~~\mathbb{S}.\label{uniform}
\end{align}
It follows that
\begin{align}
|\gamma(t,\xi)|\leq|y(t,\xi)-\xi|\leq |\bar{y}(\xi)-\xi|+\int_0^tU^2(\tau,\xi){\ud}\tau\leq h+T\frac{\big(B+|\tilde{F}_0|\big)^2}{(1-h)^2},~~~~~~~a.e.~~~\mathbb{S}.\label{uniformm}
\end{align}

Combing \eqref{vw}--\eqref{uniformm}, we deduce that the solution $\big(\gamma(t),U(t),V(t),W(t),Q(t)\big)$ remains bounded on any bounded time interval $[0,T]$ in $[W^{1,\infty}({\mathbb{S}})]^2\times[L^\infty({\mathbb{S}})]^3\subset Y$. This proves the theorem.
\end{proof}
\begin{rema}\label{global}
Similarly, for $u_0\in H^1(\mathbb{S})$, we know that system \eqref{semi} with initial data $(y_0,U_0,V_0,W_0,Q_0)$ as well has a global solution $\big(y(t),U(t),V(t),W(t),Q(t)\big)$ in $X$ for any time $t\geq0$.
\end{rema}

\section{Global admissible conservative weak solution for the original equation}
\par
In this section, we are going to prove the global existence of conservative weak solution to \eqref{vis}.

\begin{theo}\label{th4}
Let $u_0(x)\in H^{1}(\mathbb{S})$. Then the problem \eqref{vis} has a admissable conservative weak solution in the sense of Definition \ref{defi3}.
\end{theo}
\begin{proof}
From Remark \ref{global}, we get a global solution $\big(y,U,V,W,Q\big)$ to system \eqref{semi}. Hence, for each fixed $\xi\in{\mathbb{S}}$, the map $t\mapsto y(t,\xi)$ gives a solution to the following problem
\begin{equation}\label{ori}
\frac{\ud}{{\ud}t}y(t,\xi)=-U^2(t,\xi),\quad y(0,\xi)=y_0(\xi).
\end{equation}

Write
\begin{align}\label{set}
u(t,x)=U(t,\xi)\quad \text{if}\quad x=y(t,\xi).
\end{align}
We have to explain the definition makes sense. Indeed, by Theorem \ref{locall1} we deduce that $y_{\xi}(t,\xi)\geq0$ for all $t\geq0$ and a.e.~$\xi$. Therefore, the map $\xi\mapsto y(t,\xi)$ is nondecreasing. If $\xi_1<\xi_2$ but $y(t,\xi_1)=y(t,\xi_2)$, we have
\begin{equation*}
0=\int_{\xi_1}^{\xi_2}y_{\xi}(t,\eta){\ud}\eta=\int_{\xi_1}^{\xi_2}(QV)(t,\eta){\ud}\eta.
\end{equation*}
If $Q\neq0$, we discover $V=0$ in $[\xi_1,\xi_2]$, which implies $W=0$ in $[\xi_1,\xi_2]$.
It follows that
\begin{equation*}
U(t,\xi_2)-U(t,\xi_1)=\int_{\xi_1}^{\xi_2}U_{\xi}(\eta){\ud}\eta=\int_{\xi_1}^{\xi_2}(WQ)(\eta){\ud}\eta=0.
\end{equation*}
Otherwise, if $Q=0$, the above equality also makes sense.
Hence, the map $(t,x)\mapsto u(t,x)$ is well-defined for any $t\geq0$ and $x\in{\mathbb{S}}$.

From \eqref{bound1} and \eqref{set}, we give
\begin{equation}\label{ux}
u_x(t,y(t,\xi))=\frac{W}{V},\quad\quad \text{as $y_{\xi}\neq0$}.
\end{equation}
Changing the variables and applying \eqref{bound1} and \eqref{ux}, we find
\begin{align}
E(t)&=\int_{\mathbb{S}}u_x^2(t,x){\ud}x=\int_{{\mathbb{S}}\cap\{y_\xi\neq0\}}u_x^2(t,y(t,\xi))y_\xi{\ud}\xi\notag\\
&=\int_{{\mathbb{S}}\cap\{y_\xi\neq0\}}\big(Q-VQ\big)(t,\xi){\ud}\xi=\int_{\mathbb{S}}\big(Q-VQ\big)(t,\xi){\ud}\xi\notag\\
&=\tilde{E}(t)=\tilde{E}_0=\int_{\mathbb{S}}u_{0x}^2{\ud}x.\label{conn1}
\end{align}
Similarly, we gain
\begin{align}
F(t)&=\int_{\mathbb{S}}(u-uu_x^2)(t,x){\ud}x=\int_{{\mathbb{S}}\cap\{y_\xi\neq0\}}(u-uu_x^2)(t,y(t,\xi))y_\xi{\ud}\xi\notag\\
&=\int_{{\mathbb{S}}\cap\{y_\xi\neq0\}}\big(2UVQ-UQ\big)(t,\xi){\ud}\xi=\int_{\mathbb{S}}\big(2UVQ-UQ\big)(t,\xi){\ud}\xi\notag\\
&=\tilde{F}(t)=\tilde{F}_0=\int_{\mathbb{S}}u_0(1-u_{0x}^2){\ud}x.\label{conn2}
\end{align}

From the hypothesis of Theorem \ref{th}, we know that $\int_{\mathbb{S}}u_0(1-u_{0x}^2){\ud}x=0$, which implies $P(1)=0=F(t)$. Hence, Eq. \eqref{eq2} is equivalent to Eq. \eqref{eq1}. Thanks to \eqref{set} and \eqref{conn1}, we can infer that $u$ belongs to $L^\infty(\mathbb{R}^+;H^{1}(\mathbb{S}))$. We also have to prove that $u$ satifies Eq. \eqref{eq1}. In light of \eqref{semi1}, we discover $y_{\xi t}=U_{\xi}(t,\xi)$. Therefore, for any $\psi\in C_0^\infty(\mathbb{R}^+;D({\mathbb{S}}))$, applying the change of variables, we see
\begin{align}
\int_{\mathbb{R}^+}\int_{\mathbb{S}}\big(u\psi_{tx}+u^2u_x\psi_x\big)(t,x){\ud}x{\ud}t&=\int_{\mathbb{R}^+}\int_{\mathbb{S}}\big(u\psi_{tx}+u^2u_x\psi_x\big)(t,y(t,\xi))y_{\xi}{\ud}\xi{\ud}t\notag\\
&=\int_{\mathbb{R}^+}\int_{\mathbb{S}}U\psi_{tx}(t,y(t,\xi))y_{\xi}+U^2U_\xi\psi_x(t,y(t,\xi)){\ud}\xi{\ud}t\notag\\
&=\int_{\mathbb{R}^+}\int_{\mathbb{S}}U\big(\psi(t,y(t,\xi))\big)_{\xi t}+\Big(U^3\big(\psi_x(t,y(t,\xi))\big)\Big)_{\xi}{\ud}\xi{\ud}t\notag\\
&=-\int_{\mathbb{R}^+}\int_{\mathbb{S}}U_{\xi}\big(\psi(t,y(t,\xi))\big)_t{\ud}\xi{\ud}t\notag\\
&=\int_{\mathbb{R}^+}\int_{\mathbb{S}}U_{\xi t}\psi(t,y(t,\xi)){\ud}\xi{\ud}t+\int_{\mathbb{S}}\bar{U}_\xi(\xi)\psi(\bar{y}(\xi)){\ud}\xi\notag\\
&=\int_{\mathbb{R}^+}\int_{\mathbb{S}}\big(2UQV-UQ\big)\psi(t,y(t,\xi)){\ud}\xi{\ud}t+\int_{\mathbb{S}}\bar{u}_x(\bar{y}(\xi))\bar{y}_{\xi}\psi(\bar{y}(\xi)){\ud}\xi\notag\\
&=\int_{\mathbb{R}^+}\int_{\mathbb{S}}\big(u-uu_x^2\big)\psi(t,x){\ud}x{\ud}t+\int_{\mathbb{S}}\bar{u}_x(x)\psi(0,x){\ud}x,\label{weak}
\end{align}
where we use
\begin{equation*}
\big(\psi_x(t,y(t,\xi))y_\xi\big)_t=\psi_{xt}(t,y(t,\xi))y_\xi-U^2\big(\psi_x(t,y(t,\xi))\big)_\xi-2UU_\xi\psi_x(t,y(t,\xi))
\end{equation*}
in the third equality.

Similarilily,  we have
\begin{align}\label{weak1}
\int_{\mathbb{R}^+}\int_{\mathbb S}\big(u\psi_{t}-\frac{1}{3}u^3\psi_x\big)(t,x){\ud}x{\ud}t=-\int_{\mathbb{R}^+}\int_{\mathbb S} H(t,x)\psi(t,x){\ud}x{\ud}t-\int_{\mathbb{S}}u_0(x)\psi(0,x){\ud}x
\end{align}
and
\begin{align}\label{weak2}
\int_{\mathbb{R}^+}\int_{\mathbb S}\big(u^2_x\psi_{t}-u^2u^2_x\psi_{x}\big)(t,z){\ud}z{\ud}t=-\int_{\mathbb{R}^+}\int_{\mathbb S}\big[((u^2)_x-2P(1)u_{x})\psi\big](t,z){\ud}z{\ud}t-\int_{\mathbb{S}}u^2_{0x}(z)\psi(0,z){\ud}z
\end{align}
where $H(t,x)=\partial_x^{-1}(u-uu_x^2)-f(t)$ is a bound variable.

In a word, we verify that $u$ is indeed a global admissible conservative solution to the Cauchy problem \eqref{vis} in the sense of Definition \ref{defi3}. This completes the proof of Theorem \ref{th4}.

\end{proof}

\section{Uniqueness of the global admissible conservative weak solution}
\par
We now prove the uniqueness of the global admissible conservative weak solutions to \eqref{vis}:
\begin{theo}\label{th5}
Let $u(t,x)$ be a global admissible conservative weak solutions to the problem \eqref{vis} in the sense of Definition \ref{defi3}, then $u(t,x)$ is unqiue.
\end{theo}

\subsection{Useful lemmas}
\par

Since $u(t,x)$ be a global admissible conservative weak solution in the sense of Definition \ref{defi3}, we can easily deduce that $\|u\|_{H^1}\leq C\|u_0\|_{H^1}$.
We first consider the Cauchy problem:
\begin{align}\label{flow1}
\frac{{\ud}}{{\ud}t}y(t)=-u^2(t,y(t)),\quad\quad y(0)=y_0(\xi),
\end{align}
where $y_0(\xi)$ solves the equation
$y_0(\xi)+\int_{0}^{y_0(\xi)}u^2_{0x}(z){\ud}z=(1+h)\xi$.

For smooth case, combing \eqref{cons} and \eqref{flow1} we can easily deduce that
\begin{align}\label{flow1.5}
\frac{{\ud}}{{\ud}t}\int_{0}^{y(t)}u_x^2(t,z){\ud}z=\int_{0}^{y(t)}\left((u^2)_x-2P(1)u_x\right)(t,z){\ud}z-(u^2u_x^2)(t,0), \quad\quad y(0)=y_0(\xi).
\end{align}
However, in the weak sense, we must choose some special test functions to solve \eqref{flow1}. The key idea is combining \eqref{flow1} and \eqref{flow1.5} in the weak sense to get a unique solution of \eqref{flow1}.  

Instead of the variables $(t,x)$, it is convenient to work with an adapted set of variables $(t,\beta)$, where $\beta$ is implicitly defined as
\begin{align}\label{flow2}
y(t,\beta)+\int_{0}^{y(t,\beta)} u^2_x(t,z){\ud}z=(1+h)\beta.
\end{align}

Next, we present some useful lemmas.
\begin{lemm}\cite{czwhw}\label{lip00}
Let the map $x\mapsto \phi(x)$ be an absolutely continuous from $[a,b]$ to $[c,d]$. Moreover, suppose $\phi(x)$ is strictly monotonic. Then for any  measurable set $A\subset [a,b]$, we have ${\rm{meas}}(\phi(A))=\int_{A}\phi_x{\ud}x$.
\end{lemm}

\begin{lemm}\label{lip0}
Let the map $x\mapsto f(x)$ be a bijection from $[a,b]$ to $[c,d]$, moreover $f(x)$ is absolutely continuous and strictly monotonic. Then for any $E\subset [c,d]$ and ${\rm{meas}}(E)=d-c$, we have ${\rm{meas}}(f^{-1}(E))=b-a$.
\end{lemm}
\begin{proof}
We consider the complement set of $E$. If ${\rm{meas}}(f^{-1}(E^c))=\delta>0$, by lemma \ref{lip00}, let $A=f^{-1}(E^c)$ and $\phi(x)=f(x)$, we have
\begin{align}
0={\rm{meas}}(E^c)=\int_{E^c} {\ud}y=\int_{f^{-1}(E^c)}f_x{\ud}x
\end{align}
This means $f_x=0\ a.e.\ f^{-1}(E^c)$, which is in contradiction with the strict monotonicity. Therefore, we have ${\rm{meas}}(f^{-1}(E^c))=0$, and thus ${\rm{meas}}(f^{-1}(E))=b-a$.
\end{proof}

\begin{lemm}\label{lip1}
Let $u=u(t,x)$ be a global admissible conservative weak solution of \eqref{vis}. Then, for every $t\geq \tau>0$, we have
\begin{align*}
\left|{\lim_{\epsilon \to 0}}\int_{\tau}^{t}\int_{\frac{1}{8}\epsilon}^{\frac{2}{8}\epsilon}\frac{8}{\epsilon}(u^2u^2_x)(s,z){\ud}z{\ud}s,~{\lim_{\epsilon \to 0}}\int_{0}^{t}\int_{\frac{1}{8}\epsilon}^{\frac{2}{8}\epsilon}\frac{8}{\epsilon}(u^2u^2_x)(s,z){\ud}z{\ud}s,~{\lim_{\epsilon \to 0}}\int_{0}^{\tau}\int_{\frac{1}{8}\epsilon}^{\frac{2}{8}\epsilon}\frac{8}{\epsilon}(u^2u^2_x)(s,z){\ud}z{\ud}s\right| \leq C_{u_0,T}
\end{align*}
and
\begin{align*}
{\lim_{\epsilon \to 0}}\int_{\tau}^{t}\int_{\frac{1}{8}\epsilon}^{\frac{2}{8}\epsilon}\frac{8}{\epsilon}(u^2u^2_x)(s,z){\ud}z{\ud}s={\lim_{\epsilon \to 0}}\int_{0}^{t}\int_{\frac{1}{8}\epsilon}^{\frac{2}{8}\epsilon}\frac{8}{\epsilon}(u^2u^2_x)(s,z){\ud}z{\ud}s-{\lim_{\epsilon \to 0}}\int_{0}^{\tau}\int_{\frac{1}{8}\epsilon}^{\frac{2}{8}\epsilon}\frac{8}{\epsilon}(u^2u^2_x)(s,z){\ud}z{\ud}s.
\end{align*}
\end{lemm}
\begin{proof}
For $\epsilon>0$ small enough, let
\begin{equation}\label{flow2.2}
p_{\epsilon}(s,z)=\left\{
\begin{array}{rcl}
0 & & {0 \leq z<\frac{1}{8}\epsilon},\\
 \frac{8}{\epsilon}(z-\frac{\epsilon}{8}) & & {\frac{1}{8}\epsilon \leq z < \frac{2}{8}\epsilon},\\
 1 & & {\frac{2}{8}\epsilon \leq z < \frac{3}{8}\epsilon},\\
 \frac{1}{1-\frac{\epsilon}{2}}\big(1-\frac{\epsilon}{8}-z\big) & & {\frac{3}{8}\epsilon \leq z<1-\frac{1}{8}\epsilon},\\
 0 & & {1-\frac{1}{8}\epsilon \leq z<1},
\end{array} \right.
\end{equation}
\begin{equation}\label{flow2.3}
\chi_{\epsilon}(s)=\left\{
\begin{array}{rcl}
0 & & { 0 \leq s< \tau-\epsilon},\\
\frac{1}{\epsilon}(s-\tau+\epsilon) & & {\tau-\epsilon \leq s < \tau},\\
1 & & {\tau \leq s < t},\\
1-\frac{1}{\epsilon}(s-t) & & {t \leq s<t+\epsilon},\\
0 & & {t+\epsilon\leq s}.
\end{array} \right.
\end{equation}
Define
\begin{align}\label{flow2.4}
\psi_{\epsilon}(s,z):=\min \{p_{\epsilon}(s,z), \chi_{\epsilon}(s)\}.
\end{align}
By an approximation argument, the identity \eqref{weak2} remains valid for any test function $\psi$ which is Lipschitz continuous with compact support. Using $\psi_{\epsilon}$ as the test function in \eqref{weak2} we obtain
\begin{align}\label{flow2.5}
\int_{\mathbb{R}^+}\int_{\mathbb S}\big(u^2_x\psi_{\epsilon t}-u^2u^2_x\psi_{\epsilon x}\big)(s,z){\ud}z{\ud}s=-\int_{\mathbb{R}^+}\int_{\mathbb S}\big(((u^2)_x-2P(1)u_{x})\psi_{\epsilon}\big)(s,z){\ud}z{\ud}s.
\end{align}
Taking the limit of \eqref{flow2.5} as $\epsilon \rightarrow 0$, we see that
\begin{align*}
&-{\lim_{\epsilon \to 0}}\int_{\tau}^{t}\int_{\frac{1}{8}\epsilon}^{\frac{2}{8}\epsilon}\frac{8}{\epsilon}(u^2u^2_x)(s,z){\ud}z{\ud}s+{\lim_{\epsilon \to 0}}\int_{\tau}^{t}\int_{\frac{3}{8}\epsilon}^{1-\frac{1}{8}\epsilon}\frac{1}{1-\frac{\epsilon}{2}}(u^2u^2_x)(s,z){\ud}z{\ud}s\\
=&-{\lim_{\epsilon \to 0}}\int_{\tau}^{t}\int_{\frac{3}{8}\epsilon}^{1-\frac{1}{8}\epsilon}\frac{1}{1-\frac{\epsilon}{2}}\big(1-\frac{\epsilon}{8}-z\big)\big[(u^2)_x-2P(1)u_x\big](s,z){\ud}z{\ud}s.
\end{align*}
After careful calculations we have
\begin{align*}
&{\lim_{\epsilon \to 0}}\int_{\tau}^{t}\int_{\frac{3}{8}\epsilon}^{1-\frac{1}{8}\epsilon}\frac{1}{1-\frac{\epsilon}{2}}(u^2u^2_x)(s,z){\ud}z{\ud}s=\int_{\tau}^{t}\int_{0}^{1}u^2u^2_x{\ud}z{\ud}s,\\
&{\lim_{\epsilon \to 0}}\int_{\tau}^{t}\int_{\frac{3}{8}\epsilon}^{1-\frac{1}{8}\epsilon}\frac{1}{1-\frac{\epsilon}{2}}\big(1-\frac{\epsilon}{8}-z\big)\big[(u^2)_x-2P(1)u_x\big](s,z){\ud}z{\ud}s=\int_{\tau}^{t}\int_{0}^{1}(1-z)\big[(u^2)_x-2P(1)u_x\big]{\ud}z{\ud}s.
\end{align*}
Therefore, 
\begin{align}\label{ttau}
{\lim_{\epsilon \to 0}}\int_{\tau}^{t}\int_{\frac{1}{8}\epsilon}^{\frac{2}{8}\epsilon}\frac{8}{\epsilon}(u^2u^2_x)(s,z){\ud}z{\ud}s&=\int_{\tau}^{t}\int_{0}^{1}(1-z)\big((u^2)_x-2P(1)u_{x}\big)(s,z){\ud}z{\ud}s+\int_{\tau}^{t}\int_{0}^{1}u^2u^2_x(s,z){\ud}z{\ud}s,
\end{align}
which follows that 
\begin{align*}
\left|{\lim_{\epsilon \to 0}}\int_{\tau}^{t}\int_{\frac{1}{8}\epsilon}^{\frac{2}{8}\epsilon}\frac{8}{\epsilon}(u^2u^2_x)(s,z){\ud}z{\ud}s\right| \leq C_{u_0,T}.
\end{align*}

For the same $\epsilon>0$, let
\begin{equation}\label{flowp}
p_{\epsilon}(s,z)=\left\{
\begin{array}{rcl}
0 & & {0 \leq z<\frac{1}{8}\epsilon},\\
\frac{8}{\epsilon}(z-\frac{\epsilon}{8}) & & {\frac{1}{8}\epsilon \leq z < \frac{2}{8}\epsilon},\\
1 & & {\frac{2}{8}\epsilon \leq z < \frac{3}{8}\epsilon},\\
\frac{1}{1-\frac{\epsilon}{2}}\big(1-\frac{\epsilon}{8}-z\big) & & {\frac{3}{8}\epsilon \leq z<1-\frac{1}{8}\epsilon},\\
0 & & {1-\frac{1}{8}\epsilon \leq z<1},
\end{array} \right.
\end{equation}
\begin{equation}\label{flowx}
\chi_{1\epsilon}(s)=\left\{
\begin{array}{rcl}
1 & & { 0 \leq s< t},\\
1-\frac{1}{\epsilon}(s-t) & & {t \leq s<t+\epsilon},\\
0 & & {t+\epsilon\leq s}.
\end{array} \right.
\end{equation}
Define
\begin{align}
\psi_{1\epsilon}(s,z):=\min\{p_{\epsilon}(s,z),\chi_{1\epsilon}(s)\}.
\end{align}
Similarly, using $\psi_{1\epsilon}$ as the test function in \eqref{weak2} and taking the limit as $\epsilon\rightarrow 0$, we see
\begin{align}\label{t}
{\lim_{\epsilon \to 0}}\int_{0}^{t}\int_{\frac{1}{8}\epsilon}^{\frac{2}{8}\epsilon}\frac{8}{\epsilon}(u^2u^2_x)(s,z){\ud}z{\ud}s&=\int_{0}^{t}\int_{0}^{1}(1-z)\big((u^2)_x-2P(1)u_{x}\big)(s,z){\ud}z{\ud}s+\int_{0}^{t}\int_{0}^{1}u^2u^2_x(s,z){\ud}z{\ud}s.
\end{align}
In the same way, we deduce
\begin{align}\label{tau}
{\lim_{\epsilon \to 0}}\int_{0}^{\tau}\int_{\frac{1}{8}\epsilon}^{\frac{2}{8}\epsilon}\frac{8}{\epsilon}(u^2u^2_x)(s,z){\ud}z{\ud}s&=\int_{0}^{\tau}\int_{0}^{1}(1-z)\big((u^2)_x-2P(1)u_{x}\big)(s,z){\ud}z{\ud}s+\int_{0}^{\tau}\int_{0}^{1}u^2u^2_x(s,z){\ud}z{\ud}s.
\end{align}
Combining \eqref{ttau}, \eqref{t} and \eqref{tau}, we can easily get that 
\begin{align*}
{\lim_{\epsilon \to 0}}\int_{\tau}^{t}\int_{\frac{1}{8}\epsilon}^{\frac{2}{8}\epsilon}\frac{8}{\epsilon}(u^2u^2_x)(s,z){\ud}z{\ud}s={\lim_{\epsilon \to 0}}\int_{0}^{t}\int_{\frac{1}{8}\epsilon}^{\frac{2}{8}\epsilon}\frac{8}{\epsilon}(u^2u^2_x)(s,z){\ud}z{\ud}s-{\lim_{\epsilon \to 0}}\int_{0}^{\tau}\int_{\frac{1}{8}\epsilon}^{\frac{2}{8}\epsilon}\frac{8}{\epsilon}(u^2u^2_x)(s,z){\ud}z{\ud}s.
\end{align*}
This completes the proof of the lemma.
\end{proof}
Since $u^2u^2_x$ is bounded in $L^1(\mathbb{S})$, we deduce that the Lebesgue points of $u^2u^2_x$ are almost everywhere. Without loss of generality, we can assume $0$ is one of the Lebesgue points of $u^2u^2_x$, so we have
\begin{align}\label{flow5.11}
\left|{\lim_{\epsilon \to 0}}\int_{\tau}^{t}\int_{\frac{1}{8}\epsilon}^{\frac{2}{8}\epsilon}\frac{8}{\epsilon}\big(u^2u^2_x\big)(s,z){\ud}z{\ud}s:=\int_{\tau}^{t}\big(u^2u^2_x\big)(s,0){\ud}s\right|\leq C_{u_0,T}.
\end{align}

\begin{lemm}\label{lip1.1}
Let $u=u(t,x)$ be a global admissible conservative weak solution of \eqref{vis}. Then, for every $t\geq 0$, the maps $\beta\mapsto y(t,\beta)$ and $\beta \mapsto u(t,y(t,\beta))$ implicitly defined by \eqref{flow2} are Lipschitz continuous with constant $1+h$. 
\end{lemm}
\begin{proof}
1. For fixed time $t\geq 0$, it's easy to deduce that $y(t,\beta)$ is continuous and strictly monotonic. Moreover, we have $y(t,0)=0$, $y(t,1)=1$ and $y(t,\beta+1)=y(t,\beta)+1$. If $\beta_1<\beta_2$, then
\begin{align}\label{flow3}
y(t,\beta_2)-y(t,\beta_1)=-\int_{y(t,\beta_1)}^{y(t,\beta_2)} u^2_x(z){\ud}z+(1+h)(\beta_2-\beta_1)\leq (1+h)(\beta_2-\beta_1).
\end{align}

2. To prove the Lipschitz continuous of the map $\beta \mapsto u(t,y(t,\beta))$, let's suppose that $\beta_1<\beta_2$. By \eqref{flow3}, it follows that
\begin{align}\label{flow4}
\big|u(t,y(t,\beta_2))-u(t,y(t,\beta_1))\big|&\leq \int_{y(t,\beta_1)}^{y(t,\beta_2)}|u_x|{\ud}z\leq \int_{y(t,\beta_1)}^{y(t,\beta_2)}\frac{1}{2}(1+u^2_x){\ud}z \notag\\
&\leq \frac{1}{2}\Big[y(t,\beta_2)-y(t,\beta_1)+\int_{y(t,\beta_1)}^{y(t,\beta_2)} u^2_x(z){\ud}z\Big] \notag\\
&\leq \frac{1}{2}(1+h)(\beta_2-\beta_1).
\end{align}

\end{proof}

\begin{lemm}\label{lip2}
Let $u=u(t,x)$ be a global admissible conservative weak solution of \eqref{vis}. Then, for any $y_0(\xi)$ satisfying $y_0(\xi)+\int_{0}^{y_0(\xi)}u^2_{0x}(z){\ud}z=(1+h)\xi$, there exists a unique Lipschitz continuous map $t \mapsto y(t):=y(t,\beta(t,\xi))$ which satisfies \eqref{flow1}.
\end{lemm}
\begin{proof}
1. Since $u(t,x)$ belongs to $L^\infty(\mathbb{R}^+;H^1)$ and $H^1\hookrightarrow C^{\frac{1}{2}}$, then $|u(t,x)|\leq M$. For any  $b>0$, suppose that $a>0$ is small enough such that $a\leq\frac{b}{M^2}$, then $u(t,x)\in L^\infty\big([0,a];C[y_0(\xi)-b,y_0(\xi)+b]\big)$. Applying the Arezela-Ascoli theorem and the Schauder fixed point theorem, we can obtain that the Cauchy problem \eqref{flow1} has a solution $y(t)$ on $[0,a]$. Moreover, $y(t)$ is Lipschitz continuous with $t$ on $[0,a]$. For other cases, we can use the continuous method. Thus, $y(t)$ is Lipschitz continuous with time $t$ for any $t\geq 0$.
	
2. For $\epsilon>0$ small enough, define
\begin{align}\label{uyt}
	\psi_{2\epsilon }(s,z):=\min \{p_{2\epsilon}(s,z), \chi_{\epsilon}(s)\},
\end{align}
where
\begin{equation}\label{pzs}
	p_{2\epsilon}(s,z)=\left\{
	\begin{array}{rcl}
		0 & & {0 \leq z<\frac{1}{8}\epsilon},\\
		\frac{8}{\epsilon}(z-\frac{\epsilon}{8}) & & {\frac{1}{8}\epsilon \leq z < \frac{2}{8}\epsilon},\\
		1 & & \frac{2}{8}\epsilon \leq z < \frac{3}{8}\epsilon+y(s),\\
		1-\frac{8}{\epsilon}\Big(z-\frac{3}{8}\epsilon-y(s)\Big) & & {\frac{3}{8}\epsilon+y(s) \leq z<\frac{4}{8}\epsilon+y(s)},\\
		0 & & {\frac{4}{8}\epsilon+y(s) \leq y<1},
	\end{array} \right.
\end{equation}
\begin{equation}\label{flow2.33}
	\chi_{\epsilon}(s)=\left\{
	\begin{array}{rcl}
		0 & & { 0 \leq s< \tau-\epsilon},\\
		\frac{1}{\epsilon}(s-\tau+\epsilon) & & {\tau-\epsilon \leq s < \tau},\\
		1 & & {\tau \leq s < t},\\
		1-\frac{1}{\epsilon}(s-t) & & {t \leq s<t+\epsilon},\\
		0 & & {t+\epsilon\leq s}.
	\end{array} \right.
\end{equation}
By an approximation argument, the identity \eqref{weak2} remains valid for any test function $\psi$ which is Lipschitz continuous with compact support. Using $\psi_{2\epsilon}$ as the test function in \eqref{weak2} and taking the limit as $\epsilon \rightarrow 0$, we find 
\begin{align}\label{psi2}
	\int_{0}^{y(t)} u^2_x(z){\ud}z-\int_{0}^{y(\tau)} u^2_x(z){\ud}z=&\int_{\tau}^{t}\int_{0}^{y(s)}\big((u^2)_x-2P(1)u_{x}\big){\ud}z{\ud}s-{\lim_{\epsilon \to 0}}\int_{\tau}^{t}\int_{\frac{1}{8}\epsilon}^{\frac{2}{8}\epsilon}\frac{8}{\epsilon}(u^2u^2_x){\ud}z{\ud}s\notag\\
	&-{\lim_{\epsilon \to 0}}\int_{\tau}^{t}\int_{\frac{3}{8}\epsilon+y(s)}^{\frac{4}{8}\epsilon+y(s)}\frac{8}{\epsilon}u^2_x\Big(u^2(s,y(s))-u^2(s,z)\Big){\ud}z{\ud}s.
\end{align}
Since $u_x\in L^2(\mathbb{S})$, similar to \cite{bcz} by the Cauchy inequality and the dominated convergence theorem, we deduce that
\begin{align*}
	{\lim_{\epsilon \to 0}}\int_{\tau}^{t}\int_{\frac{3}{8}\epsilon+y(s)}^{\frac{4}{8}\epsilon+y(s)}\frac{8}{\epsilon}u^2_x\Big(u^2(s,y(s))-u^2(s,z)\Big){\ud}z{\ud}s=0\quad \text{for a.e. $y(s)\in\mathbb{S}$.}
\end{align*}
By lemma \ref{lip0} and the fact that the map $\xi \mapsto y(t):=y(t,\beta(t,\xi))$ is strictly monotonic and Lipschitz continuous from $[0,1]$ to $[0,1]$, we see that
\begin{align*}
	{\lim_{\epsilon \to 0}}\int_{\tau}^{t}\int_{\frac{3}{8}\epsilon+y(s)}^{\frac{4}{8}\epsilon+y(s)}\frac{8}{\epsilon}u^2_x\Big(u^2(s,y(s))-u^2(s,z)\Big){\ud}z{\ud}s=0\quad \text{for a.e. $\xi\in\mathbb{S}$.}
\end{align*}
Thus, for almost everywhere $\xi\in\mathbb{S}$, we discover
\begin{align}\label{xi}
	\int_{0}^{y(t)} u^2_x(z){\ud}z-\int_{0}^{y(\tau)} u^2_x(z){\ud}z=\int_{\tau}^{t}\int_{0}^{y(s)}\big((u^2)_x-2P(1)u_{x}\big){\ud}z{\ud}s-{\lim_{\epsilon \to 0}}\int_{\tau}^{t}\int_{\frac{1}{8}\epsilon}^{\frac{2}{8}\epsilon}\frac{8}{\epsilon}(u^2u^2_x){\ud}z{\ud}s.
\end{align}
When we take $\tau=0$ in \eqref{xi}, we get that for almost everywhere $\xi\in\mathbb{S}$,
\begin{align}\label{xii}
	\int_{0}^{y(t)} u^2_x(z){\ud}z=\int_{0}^{y_0(\xi)} u^2_{0x}(z){\ud}z+\int_{0}^{t}\int_{0}^{y(s)}\big((u^2)_x-2P(1)u_{x}\big){\ud}z{\ud}s-{\lim_{\epsilon \to 0}}\int_{0}^{t}\int_{\frac{1}{8}\epsilon}^{\frac{2}{8}\epsilon}\frac{8}{\epsilon}(u^2u^2_x){\ud}z{\ud}s.
\end{align}

3. We consider the equation:
\begin{align}\label{flow5}
(1+h)\beta(t,\xi)=&y(t)+\int_{0}^{y(t)} u^2_x(z){\ud}z\notag\\
=&y_0(\xi)+\int_{0}^{y_0(\xi)} u^2_{0x}(z){\ud}z+\int_{0}^{t}\Big[-u^2(s,y(s))+\int_{0}^{y(s)}(u^2)_x-2P(1)u_x{\ud}z\Big]{\ud}s\notag\\
&-{\lim_{\epsilon \to 0}}\int_{0}^{t}\int_{\frac{1}{8}\epsilon}^{\frac{2}{8}\epsilon}\frac{8}{\epsilon}(u^2u^2_x){\ud}z{\ud}s\notag\\
=&(1+h)\xi+\int_{0}^{t}\int_{0}^{y(s)}-2P(1)u_x{\ud}z{\ud}s-\int_{0}^{t}u^2(s,0){\ud}s-{\lim_{\epsilon \to 0}}\int_{0}^{t}\int_{\frac{1}{8}\epsilon}^{\frac{2}{8}\epsilon}\frac{8}{\epsilon}(u^2u^2_x){\ud}z{\ud}s\notag\\
=&(1+h)\xi+\int_{0}^{t}\int_{0}^{y(s)}-2P(1)u_x{\ud}z{\ud}s-\int_{0}^{t}u^2(s,0){\ud}s-A(t),
\end{align}
where $A(t)={\lim\limits_{\epsilon \to 0}}\int_{0}^{t}\int_{\frac{1}{8}\epsilon}^{\frac{2}{8}\epsilon}\frac{8}{\epsilon}(u^2u^2_x){\ud}z{\ud}s$.\\
Introducing the function
\begin{align}\label{flow6}
G(s,\beta(s,\xi))=\int_{0}^{y(s,\beta(s,\xi))}-2P(1)u_x{\ud}z.
\end{align}
Then we can rewrite the right-hand side of \eqref{flow5} in the from
\begin{align}\label{flow7}
\beta(t,\xi)=\xi+\frac{1}{1+h}\int_{0}^{t}G(s,\beta(s,\xi)){\ud}s-\frac{1}{1+h}\int_{0}^{t}u^2(s,0){\ud}s-\frac{1}{1+h}A(t).
\end{align}

For each fixed $t\geq 0$, since $\|u\|_{H_1}\leq C\|u_0\|_{H_1}$ and the map $\beta \mapsto u(t,y(t,\beta))$ is Lipschitz continuous, we can deduce that the function $\beta \mapsto G(s,\beta)$ is uniformly bounded and Lipschitz continuous:
\begin{align}\label{flow8}
|G(s,\beta_1)-G(s,\beta_2)|\leq C_{u_0}|\beta_1-\beta_2|.
\end{align}

Moreover, for each fixed $t\geq 0$, the map $\xi \mapsto \beta(t,\xi)$ is strictly monotonic and Lipschitz continuous.
Indeed, by \eqref{flow8} we have
\begin{align}\label{flow8.}
|\beta(t,\xi_2)-\beta(t,\xi_1))|&\leq |\xi_2-\xi_1|+\frac{1}{1+h}\int_{0}^{t}|G(s,\beta(s,\xi_2))-G(s,\beta(s,\xi_1))|{\ud}s\notag\\
&\leq |\xi_2-\xi_1|+C_{u_0}\int_{0}^{t}|\beta(s,\xi_2)-\beta(s,\xi_1)|{\ud}s\notag\\
&\leq e^{C_{u_0}T}|\xi_2-\xi_1|,
\end{align}
where the last inequality is obtained by using the Gronwall lemma.

Assuming $\xi_2 >\xi_1 $, we have
\begin{align}\label{flow8..}
\beta(t,\xi_2)-\beta(t,\xi_1)= \xi_2-\xi_1+\frac{1}{1+h}\int_{0}^{t}G(s,\beta(s,\xi_2))-G(s,\beta(s,\xi_1)){\ud}s \geq (\xi_2-\xi_1)(1-C_{u_0}t).
\end{align}
This implies that monotonicity makes sense when $t$ is sufficiently small. Without loss of generality, we can assume that $t$ is sufficiently small, otherwise we can use the continuous method. For each fixed $t\geq 0$, being the composition of two Lipschitz functions, then the maps $\xi \mapsto G(t,y(t,\beta(t,\xi)))$ and $\xi \mapsto u(t,y(t,\beta(t,\xi)))$ are Lipschitz continuous.

4. Thanks to the Lipschitz continuity of the function G, the existence of a unique solution to the integral equation \eqref{flow7} can be prove by a standard fixed point theory. As a matter of fact, we consider the Banach space of all continuous function $\beta:\mathbb{R}^+\rightarrow \mathbb{R} $ with weighted norm
$$|\|\beta\||:=\sup_{t\geq0}e^{-2Ct}|\beta(t)|.$$
On this space, we claim that the Picard map
$$(P\beta)(t)=\xi+\int_{0}^{t}G(s,\beta(s,\xi)){\ud}s-\int_{0}^{t}u^2(s,0){\ud}s-A(t)$$
ia a strict contraction. Indeed, assume $|\|\beta_2-\beta_1\||=a>0$. This implies
$$|\beta_2(s)-\beta_1(s)|\leq ae^{2Cs}.$$
Hence
\begin{align}
|(P\beta_2)(t)-(P\beta_1)(t)|=&\Big|\int_{0}^{t}G(s,\beta_1(s))-G(s,\beta_2(s)){\ud}s\Big|\leq C\int_{0}^{t}\big|\beta_1(s)-\beta_2(s)\big|{\ud}s\notag\\
\leq& \int_{0}^{t}Cae^{2Cs}{\ud}s\leq \frac{a}{2}e^{2Ct}.
\end{align}
We thus conclude that $|\|P\beta_2-P\beta_1\||\leq \frac{a}{2}$, by the contraction mapping principle, the integral equation \eqref{flow7} has a unique solution.

Since the map $\beta \mapsto y(t,\beta)$ is strictly monotonic increasing and Lipschitz continuous, we obtain a unique solution $y(t):=y(t,\beta(t,\xi))$ from \eqref{flow5}. Being the composition of two Lipschitz functions, the map $t \mapsto y(t,\beta(t,\xi))$ is Lipschitz continuous and provides the unique solution to \eqref{flow5}. 

5. Finally, we prove the uniqueness of \eqref{flow1}. Assume there are two different functions $y_1(\cdot)$ and $y_2(\cdot)$, both satisfying \eqref{xii} together with the characteristic
equation \eqref{flow1}. Choose the  measurable functions $\beta_1$ and $\beta_2$ such that $y_1(t)=y(t,\beta_1(t))$ and $y_2(t)=y(t,\beta_2(t))$. Then $y_1(t)$ and $y_2(t)$ satisfy \eqref{flow5} with the same initial condition. This contradicts the uniqueness of \eqref{flow5} proved in step 4. 
\end{proof}

\begin{lemm}\label{lip3}
For any $0\leq \tau \leq t$, the map $t \mapsto u(t,y(t))$ is Lipschitz continuous and we have
\begin{align}\label{flow20}
u(t,y(t))-u(\tau,y(\tau))=\int_{\tau}^{t} H(s,y(s)){\ud}s,
\end{align}
where $H(s,x)=\partial_x^{-1}(u-uu_x^2)-f(s)$ is a bound variable.
\end{lemm}
\begin{proof}
By \eqref{weak1}, for any test function $\phi \in C_{0}^\infty(\mathbb{R}^+;\mathcal{D}(\mathbb S))$, one has
\begin{align}\label{flow20.5}
\int_{\mathbb{R}^+}\int_{\mathbb S}\big(u\phi_{t}-\frac{1}{3}u^3\phi_{x}\big)(s,z){\ud}z{\ud}s=-\int_{\mathbb{R}^+}\int_{\mathbb S}\big(H\phi\big)(s,z){\ud}z{\ud}s-\int_{\mathbb S}u_0(z)\phi(0,z){\ud}z.
\end{align}
Given any $\psi \in C_{0}^\infty(\mathbb{R}^+;\mathcal{D}(\mathbb S))$, let $\phi=\psi_x$. Since the map $x\mapsto u(t,x)$ is absolutely continuous, integrating by parts, we can obtain
\begin{align}\label{flow21}
\int_{\mathbb{R}^+}\int_{\mathbb S}\big(u_x\psi_{t}-u^2u_x\psi_{x}\big)(s,z){\ud}z{\ud}s=\int_{\mathbb{R}^+}\int_{\mathbb S}\big(H\psi_x\big)(s,z){\ud}z{\ud}s-\int_{\mathbb S}u_{0x}(z)\psi(0,z){\ud}z.
\end{align}
By an approximation argument, the identity \eqref{flow21} remains valid for any test function $\psi$ if it is Lipschitz continuous with compact support. Then we first let $\psi=\psi_{2\epsilon}$ in \eqref{uyt}.  Taking the limit of \eqref{flow21} as $\epsilon \rightarrow 0$, we have
\begin{align}\label{flow22.}
\int_{0}^{y(t)}u_x(t,z){\ud}z-\int_{0}^{y(\tau)}u_x(\tau,z){\ud}z=&\int_{\tau}^{t} H(s,y(s))-H(s,0){\ud}s-{\lim_{\epsilon \to 0}}\int_{\tau}^{t}\int_{\frac{1}{8}\epsilon}^{\frac{2}{8}\epsilon} \frac{8}{\epsilon}(u^2u_x)(s,z){\ud}z{\ud}s\notag\\
&-{\lim_{\epsilon \to 0}}\int_{\tau}^{t}\int_{\frac{3}{8}\epsilon+y(s)}^{\frac{4}{8}\epsilon+y(s)} \frac{8}{\epsilon}u_x(s,z)\Big(u^2(s,y(s))-u^2(s,z)\Big){\ud}z{\ud}s.
\end{align}
Applying Lebesgue differentiation theorem, we obtain that
\begin{align}\label{folw16.}
{\lim_{\epsilon \to 0}}\int_{\tau}^{t}\int_{\frac{3}{8}\epsilon+y(s)}^{\frac{4}{8}\epsilon+y(s)}\frac{8}{\epsilon}u_x(s,z)\Big(u^2(s,y(s))-u^2(s,z)\Big){\ud}z{\ud}s
=0\quad \text{for a.e. $y(s)\in\mathbb{S}$.}
\end{align}
By lemma \ref{lip0}  and the fact that the map $\xi \mapsto y(s):=y(s,\beta(s,\xi))$ is strictly monotonic and Lipschitz continuous from $[0,1]$ to $[0,1]$, we deduce that
\begin{align}\label{folw17}
{\lim_{\epsilon \to 0}}\int_{\tau}^{t}\int_{\frac{3}{8}\epsilon+y(s)}^{\frac{4}{8}\epsilon+y(s)}\frac{8}{\epsilon}u_x(s,z)\Big(u^2(s,y(s))-u^2(s,z)\Big){\ud}z{\ud}s=0 \quad \text{for a.e. $\xi\in\mathbb{S}$.}
\end{align}
Thus, for almost everywhere $\xi\in\mathbb{S}$,
\begin{align}\label{flow22}
\int_{0}^{y(t)}u_x(t,z){\ud}z-\int_{0}^{y(\tau)}u_x(\tau,z){\ud}z=&\int_{\tau}^{t} H(s,y(s))-H(s,0){\ud}s-{\lim_{\epsilon \to 0}}\int_{\tau}^{t}\int_{\frac{1}{8}\epsilon}^{\frac{2}{8}\epsilon} \frac{8}{\epsilon}(u^2u_x)(s,z){\ud}z{\ud}s.
\end{align}

To prove that the map $t \mapsto u(t,y(t))$ is Lipschitz continuous, we modify the  previous test function. For any $k\in[0,1]$, let
\begin{equation}\label{flow23}
p_{3\epsilon}(s,z)=\left\{
\begin{array}{rcl}
0 & & {0 \leq z<\frac{1}{8}\epsilon},\\
\frac{8}{\epsilon}(z-\frac{\epsilon}{8}) & & {\frac{1}{8}\epsilon \leq z < \frac{2}{8}\epsilon},\\
1 & & {\frac{2}{8}\epsilon \leq z < \frac{3}{8}\epsilon+k},\\
1-\frac{8}{\epsilon}(z-\frac{3}{8}\epsilon-k) & & {\frac{3}{8}\epsilon+k \leq z<\frac{4}{8}\epsilon+k},\\
0 & & {\frac{4}{8}\epsilon+k \leq z<1},
\end{array} \right.
\end{equation}
\begin{equation}\label{flow23.5}
\chi_{\epsilon}(s)=\left\{
\begin{array}{rcl}
0 & & { 0 \leq s< \tau-\epsilon},\\
\frac{1}{\epsilon}(s-\tau+\epsilon) & & {\tau-\epsilon \leq s < \tau},\\
1 & & {\tau \leq s < t},\\
1-\frac{1}{\epsilon}(s-t) & & {t \leq s<t+\epsilon},\\
0 & & {t+\epsilon\leq s}.
\end{array} \right.
\end{equation}
Define
\begin{align}\label{flow24}
\psi_{3\epsilon}(s,z):=\min \{p_{3\epsilon}(s,z), \chi_{\epsilon}(s)\}.
\end{align}

\noindent Using the test function $\psi=\psi_{3\epsilon}$ in \eqref{flow21} and taking the limit as $\epsilon \rightarrow 0$, we gain
\begin{align}\label{flow25}
\int_{0}^{k}u_x(t,z){\ud}z-\int_{0}^{k}u_x(\tau,z){\ud}z
=&\int_{\tau}^{t} H(s,k)-H(s,0){\ud}s-{\lim_{\epsilon \to 0}}\int_{\tau}^{t}\int_{\frac{1}{8}\epsilon}^{\frac{2}{8}\epsilon}\frac{8}{\epsilon}(u^2u_x)(s,z){\ud}z{\ud}s\notag\\
&+{\lim_{\epsilon \to 0}}\int_{\tau}^{t}\int_{\frac{3}{8}\epsilon+k}^{\frac{4}{8}\epsilon+k}\frac{8}{\epsilon}u^2u_x{\ud}z{\ud}s.
\end{align}
In addition, using $\phi=\psi_{3\epsilon}$ in \eqref{flow20.5} and taking the limit as $\epsilon \rightarrow 0$, we obtain
\begin{align}\label{flow26}
\int_{0}^{k}u(t,z){\ud}z-\int_{0}^{k}u(\tau,z){\ud}z=&\int_{\tau}^{t}\int_{0}^{k} H(s,z){\ud}z{\ud}s-\frac{1}{3}{\lim_{\epsilon \to 0}}\int_{\tau}^{t}\int_{\frac{1}{8}\epsilon}^{\frac{2}{8}\epsilon}\frac{8}{\epsilon}u^3{\ud}z{\ud}s\notag\\
&+\frac{1}{3}{\lim_{\epsilon \to 0}}\int_{\tau}^{t}\int_{\frac{3}{8}\epsilon+k}^{\frac{4}{8}\epsilon+k}\frac{8}{\epsilon}u^3{\ud}z{\ud}s.
\end{align}
Differentiating \eqref{flow26} with the variable $k$ and using the fact that $u$ is bounded and absolutely continuous, we have
\begin{align}\label{flow27}
u(t,k)-u(\tau,k)=&\int_{\tau}^{t} H(s,k){\ud}s+{\lim_{\epsilon \to 0}}\int_{\tau}^{t}\int_{\frac{3}{8}\epsilon+k}^{\frac{4}{8}\epsilon+k}\frac{8}{\epsilon}u^2u_x{\ud}z{\ud}s.
\end{align}
Subtracting \eqref{flow27} from \eqref{flow25}, we receive
\begin{align}\label{flow28}
u(t,0)-u(\tau,0)=\int_{\tau}^{t}H(s,0){\ud}s+{\lim_{\epsilon \to 0}}\int_{\tau}^{t}\int_{\frac{1}{8}\epsilon}^{\frac{2}{8}\epsilon}\frac{8}{\epsilon}(u^2u_x)(s,z){\ud}z{\ud}s,
\end{align}
Subtracting \eqref{flow28} from \eqref{flow22} once again, we finally get
\begin{align}\label{flow29}
u(t,y(t))-u(\tau,y(\tau))=\int_{\tau}^{t} H(s,y(s)){\ud}s.
\end{align}
Thus,
\begin{align*}
|u(t,y(t))-u(\tau,y(\tau))|=\Big|\int_{\tau}^{t} H(s,y(s)){\ud}s\Big|\leq C_{u_0}(t-\tau),
\end{align*}
that is the map $t \mapsto u(t,y(t))$ is Lipschitz continuous.
\end{proof}

\subsection{Uniqueness of the global admissible conservative weak solution}
\par

We first prove Theorem \ref{th5}. The proof will be worked out in several steps, which is similar to \cite{bcz}.

\textbf{Proof of Theorem \ref{th5}:}~1. From Lemma \ref{lip1.1}--\ref{lip3}, the maps $(t,\beta) \mapsto (y,u)(t,\beta)$, $\beta \mapsto G(t,\beta)$ and $\beta \mapsto H(t,\beta)$ are Lipschitz continuous. By Rademacher's theorem, the partial derivatives $y_t,y_{\beta},u_t,u_{\beta},G_{\beta}$ and $H_{\beta}$ exist almost everywhere. Since $t\mapsto\beta(t,\xi)$ is the unique solution to \eqref{flow5}, then the following holds: 

\textbf{(GC)}~ For $a.e.$ $\xi$ and $a.e.$ $t>0$, the point $\beta(t,\xi)$ ia a Lebesgue point for the partial derivatives $y_t,y_{\beta},u_t,u_{\beta},G_{\beta}$ and $H_{\beta}$. Moreover, $y_{\beta}(t,\xi)>0$ for $a.e.$ $t>0$.

If \textbf{(GC)} holds, we then say that $t \mapsto \beta(t,\xi)$ is a good characteristic.
\quad

2. We seek an ODE describing how the quantities $u_{\beta}$ and $y_{\beta}$ vary along a good characteristic. Assume $\tau,t \not \in\mathcal{N}$ and $\beta(\tau,t,\xi)$ be a good characteristic, where $\beta(\tau,t,\xi)$ is a more general definition of \eqref{flow7}:
\begin{align}\label{unique.1}
\beta(\tau,t,\xi)=\xi+\frac{1}{1+h}\int_{\tau}^{t}G(s,\beta(s,\xi)){\ud}s-\frac{1}{1+h}\int_{\tau}^{t}u^2(s,0){\ud}s-\frac{1}{1+h}{\lim_{\epsilon \to 0}}\int_{\tau}^{t}\int_{\frac{1}{8}\epsilon}^{\frac{2}{8}\epsilon}\frac{8}{\epsilon}(u^2u^2_x){\ud}z{\ud}s.
\end{align}
Differentiating \eqref{unique.1} with $\xi$, we find
\begin{align}\label{unique.2}
\frac{{\ud}}{{\ud}\xi}\beta(\tau,t,\xi)=1+\frac{1}{1+h}\int_{\tau}^{t}G_{\beta}(s,\beta(s,\xi))\frac{{\ud}\beta}{{\ud}\xi}{\ud}s.
\end{align}
Next, differentiating with $\xi$ the identity
\begin{align}
y(t,\beta(t,\xi))=y(\tau,\xi)-\int_{\tau}^{t}u^2(s,y(s,\beta(s,\xi))){\ud}s,
\end{align}
we have
\begin{align}\label{unique.3}
y_{\beta}(t,\beta(t,\xi))\frac{{\ud}}{{\ud}\xi}\beta(\tau,t,\xi)=y_{\xi}(\tau,\xi)-\int_{\tau}^{t}(u^2)_{\beta}(s,y(s,\beta(s,\xi)))\frac{{\ud}}{{\ud}\xi}\beta(s,t,\xi){\ud}s.
\end{align}
Finally, differentiating with $\xi$ the identity \eqref{flow20}, we get
\begin{align}\label{unique.4}
u_{\beta}(t,y(t))\frac{{\ud}}{{\ud}\xi}\beta(\tau,t,\xi)=u_{\beta}(\tau,y(\tau))+\int_{\tau}^{t} H_{\beta}(s,y(s))\frac{{\ud}}{{\ud}\xi}\beta(s,t,\xi){\ud}s.
\end{align}
Combining \eqref{unique.2}-\eqref{unique.4}, we thus obtain the following ODEs
\begin{equation}\label{unique.5}
\left\{\begin{array}{lll}
\frac{{\ud}}{{\ud}t}[\frac{{\ud}}{{\ud}\xi}\beta(\tau,t,\xi)]=\frac{1}{1+h}G_{\beta}(s,\beta(s,\xi))\frac{{\ud}\beta}{{\ud}\xi},\\
\frac{{\ud}}{{\ud}t}[y_{\beta}(t,\beta(t,\xi))\frac{{\ud}}{{\ud}\xi}\beta(\tau,t,\xi)]=-(u^2)_{\beta}(s,y(s,\beta(t,\xi)))\frac{{\ud}}{{\ud}\xi}\beta(\tau,t,\xi),\\
\frac{{\ud}}{{\ud}t}[u_{\beta}(t,y(t))\frac{{\ud}}{{\ud}\xi}\beta(\tau,t,\xi)]=H_{\beta}(s,y(s))\frac{{\ud}}{{\ud}\xi}\beta(\tau,t,\xi).
\end{array}\right.
\end{equation}
In particular, it's easy to verify that the quantities within square brackets on the left hand sides of \eqref{unique.5} are absolutely continuous. After some calculations, we find
\begin{equation}\label{unique.6}
\left\{\begin{array}{lll}
\frac{{\ud}}{{\ud}t}y_{\beta}+\frac{1}{1+h}G_{\beta}y_{\beta}=-(u^2)_{\beta},\\
\frac{{\ud}}{{\ud}t}u_{\beta}+\frac{1}{1+h}G_{\beta}u_{\beta}=H_{\beta}.
\end{array}\right.
\end{equation}

3. We now return to the original coordinate $(t,x)$ and derive an evolution equation for the partial derivative $u_x$ along a "good" characteristic curve.

Fix a point $(\tau,x)$ with $\tau \not\in \mathcal{N}$. Assume that $x$ is a Lebesgue point for the map $x\mapsto u_x(\tau,x)$. Let $\xi$ be such that $x=y(\tau,\xi)$ and assume that $t\mapsto \beta(\tau,t,\xi)$ is a good characteristic, so that \textbf{(GC)} holds. From \eqref{flow2} we observe that
\begin{align}\label{unique.7}
y_{\beta}(\tau,\beta)=\frac{1+h}{1+u^2_x(\tau,y(\tau,\beta))}>0.
\end{align}
Then the partial derivative $u_x$ can be computed as
$$u_x(t,y(t,\beta(\tau,t,\xi)))=\frac{u_{\beta}(t,y(t,\beta(\tau,t,\xi)))}{y_{\beta}(t,\beta(t,\xi))}.$$
Using \eqref{unique.3} and \eqref{unique.4} describing the evolution of $u_{\beta}$ and $y_{\beta}$, we can easily conclude that the map $t\mapsto u_x(t,y(t,\beta(\tau,t,\xi)))$ is absolutely continuous and satisfies
\begin{align}\label{unique5.53}
\frac{{\ud}}{{\ud}t}u_x(t,y(t,\beta(\tau,t,\xi)))=\frac{{\ud}}{{\ud}t}\bigg(\frac{u_{\beta}(t,y(t,\beta(\tau,t,\xi)))}{y_{\beta}(t,\beta(t,\xi))}\bigg)=\frac{y_{\beta}H_{\beta}+2uu_{\beta}^2}{y^2_{\beta}}.
\end{align}

4. Suppose that $u=u(t,x)$ is a admissible conservative weak solution, let
\begin{align*}
&U(t,\xi)=u\big(t,y(t,\xi)\big),\quad V(t,\xi)=\frac{1}{1+u_x^2\circ y},\\
&W(t,\xi)=\frac{u_x\circ y}{1+u_x^2\circ y},\quad Q(t,\xi)=(1+u_x^2\circ y)y_{\xi}.
\end{align*}
After some calculations, we get the following semilinear system
\begin{equation}\label{unique.8}
\left\{\begin{array}{lll}
y_t(t,\xi)=-U^2,\\
U_t(t,\xi)=K(\xi)-P(1)y(\xi)-\frac{1}{1-h}\int_{\mathbb{S}}\big(2QV-Q\big)(\eta)\big(K(\eta)-P(1)y(\eta)\big){\ud}\eta,\\
V_t(t,\xi)=-2UW+2WVP(1),\\
W_t(t,\xi)=2UV-U-2V^2P(1)+VP(1),\\
Q_t(t,\xi)=-2WQP(1).
\end{array}\right.
\end{equation}
For every $\xi\in \mathbb{S}$, we take the following initial condition
\begin{equation}\label{unique.9}
\left\{\begin{array}{lll}
\int_0^{y_0(\xi)}u_{0x}^2{\ud}x+y_0(\xi)=(1+h)\xi,\\
U_0(\xi)=u_0\circ y_0(\xi),\\
V_0(\xi)=\frac{1}{1+u_{0x}^2\circ y_0(\xi)},\\
W_0(\xi)=\frac{u_{0x}\circ y_0(\xi)}{1+u_{0x}^2\circ y_0(\xi)},\\
Q_0(\xi)=(1+u_{0x}^2\circ y_0)y_{0\xi}(\xi)=1+h,
\end{array}\right.
\end{equation}
By the Lipschitz continuity of all coefficients and the previous steps, the Cauchy problem \eqref{unique.8}--\eqref{unique.9} has a unique global solution.

5. To prove the uniqueness, consider two admissible conservative weak solutions $u_1$ and $u_2$ of \eqref{vis} with the same initial data $u_0\in H^1$. For $a.e.~t\geq0$ the corresponding Lipschitz continuous maps $\xi \mapsto y_1(t,\xi)$, $\xi \mapsto y_2(t,\xi)$ are strictly increasing. Hence they have continuous inverses, say $x \mapsto ({y_1})^{-1}(t,x)$, $x \mapsto ({y_2})^{-1}(t,x)$.

As we deduce that
$$y_1(t,\xi)=y_2(t,\xi),\quad\quad u_1(t,y_1(t,\xi))=u_2(t,y_2(t,\xi)).$$
In turn, for a.e. $t\geq0$ and a.e. $x \in \mathbb{S}$, this implies
$$u_1(t,x)=u_1\big(t,{y_1}(t,\xi)\big)=u_2\big(t,{y_2}(t,\xi)\big)=u_2(t,x).$$
This completes the proof of uniqueness. \hfill$\Box$

\textbf{Proof of Theorem \ref{th}~:}

Combining Theorem \ref{th4} and Theorem \ref{th5}, we obtain an unique global admissible conservative weak solution of \eqref{vis}. Since the initial data $\int_{\mathbb{S}}u_0-u_0u_{0x}^2{\ud}x=0$, we have $\int_{\mathbb{S}}u-u{u}^2_x{\ud}x=0$ by the conservative laws. Hence, we deduce that \eqref{vis} is equivalent to \eqref{eq1}, thereby we gain an unique global admissible conservative weak solution of \eqref{eq1}.  \hfill$\Box$
\quad \\

\noindent\textbf{Acknowledgements.}
This work was partially supported by NNSFC (No. 11671407), FDCT (No. 0091/2013/A3), Guangdong Special Support Program (No. 8-2015)
and the key project of NSF of Guangdong Province (No. 2016A030311004).

\bibliographystyle{plain}
\bibliography{reference}
\end{document}